\pdfoutput=1
\documentclass[11pt]{article}

\usepackage[a4paper,margin=1in]{geometry}
\usepackage[T1]{fontenc}
\usepackage{lmodern}
\usepackage{microtype}

\usepackage{amsmath,amssymb,amsfonts,amsthm}
\usepackage{mathtools}
\usepackage{graphicx}
\usepackage{booktabs}
\usepackage{enumitem}
\usepackage[hidelinks]{hyperref}
\usepackage{url}

\newtheorem{definition}{Definition}
\newtheorem{assumption}{Assumption}
\newtheorem{lemma}{Lemma}
\newtheorem{theorem}{Theorem}
\newtheorem{proposition}{Proposition}

\newcommand{\R}{\mathbb{R}}
\newcommand{\Rpos}{\mathbb{R}_{\ge 0}}

\title{Multi-Path Routing in AMM Exchange Networks:\\
  Graph-Theoretic Formulation, Convex Allocation, and Empirical Evaluation}
\author{Ilia Zhavoronkov\\
  \small Planet 9 Group Corporation}
\date{April 2026}

\begin{document}

\maketitle

\begin{abstract}
We present a graph-theoretic and convex optimization framework for multi-path
routing in decentralized exchange networks, together with its implementation
and empirical evaluation on Ethereum mainnet.
The framework models the market as a directed token multigraph whose arcs
carry AMM exchange functions.
Routing is decomposed into two implemented layers: candidate path generation
via gas-aware marginal $k$-shortest-path enumeration---where edge scores
$w_a = -\log\pi_a + \gamma g_a$ embed expected execution cost directly into
graph traversal---with an explicit pool-simple constraint tracked during path
construction, and continuous flow allocation over the selected candidates
solved as a concave maximization over a simplex with a per-pool price-impact
cap.
Under standard concavity and monotonicity assumptions, the KKT conditions
imply marginal-output equalization across active paths.
The central technical contribution is an \emph{improving-path certificate}:
after solving the allocation on $k = 20$ candidate paths, the KKT multiplier
$\lambda^*$ is used as a threshold to determine---via a single shortest-path
query---whether any omitted pool-simple path could improve the current
solution; in our implementation the certificate confirms sufficiency in the
majority of epochs.
Execution is protected by an on-chain slippage tolerance enforced at the
smart-contract level.
We evaluate the implemented engine against four production DEX aggregators on
repeated WETH--USDT quote observations across six trade sizes on Ethereum
mainnet: median shortfall is below 5~bps across all sizes and top-3 quote
rank exceeds 57\% of epochs.
\end{abstract}

\section{Introduction}

Decentralized exchange execution is fundamentally a routing problem.
Liquidity is fragmented across automated market maker pools, fee tiers,
protocols, and venues.
A user submitting a swap request does not interact with a single central limit
order book.
Instead, the execution system must decide which pools to use, in which order,
and how to split the input amount across multiple possible paths.

Recent work has approached DEX routing from several complementary directions.
Line-graph-based methods transform the token graph into a graph over directed
liquidity-pool edges and use this representation to identify profitable linear
trading paths~\cite{zhang2025linegraph}.
Subsequent extensions introduce BFS traversal rules, route splitting, and DEX
aggregator settings~\cite{zhang2025extensions}.
In parallel, empirical audits of AMM routing have shown that realized routes
can be evaluated against optimized benchmarks and that routing losses are
affected by pool activation, flow splitting, gas costs, stale information, and
adversarial execution~\cite{xi2025suboptimality}.

We address both the modeling and measurement sides of this problem.
On the theory side, we build a token multigraph whose arcs carry AMM exchange
functions, enforce pool-simple paths during candidate enumeration, and solve
flow allocation as a concave maximization with an explicit KKT certificate.
On the empirical side, we run synchronized quote requests to five production
aggregators across six trade sizes and measure rank, tracking ratio, and
shortfall relative to the best observed quote in each epoch.

\subsection{Contributions}

The paper makes four contributions, each with a corresponding implementation.

First, it defines a directed token multigraph formulation for AMM routing in
which venues induce parallel arcs carrying nonlinear exchange functions.
Pool-simple paths---paths that do not reuse the same liquidity pool---are
enforced during path construction via explicit bookkeeping, not as a
post-hoc filter.

Second, it formulates multi-path routing over a finite candidate set as a
path-separable concave allocation problem and implements continuous flow
allocation over $k = 20$ candidate paths, with a per-pool price-impact cap
applied during allocation to bound slippage on any single liquidity source.

Third, it introduces and implements a gas-aware candidate generation layer
and an improving-path certificate.
Edge scores $w_a = -\log\pi_a + \gamma g_a$ embed expected gas cost directly
into the graph traversal, so paths with poor net output are suppressed at
enumeration time rather than filtered afterwards.
After solving the allocation, the KKT multiplier $\lambda^*$ is used as a
threshold: a single shortest-path query determines whether any omitted
pool-simple path could improve the current solution.
In practice, the certificate confirms candidate-set sufficiency in the
majority of routing epochs.

Fourth, it provides a quote-level empirical evaluation on Ethereum mainnet
using repeated WETH--USDT quote sampling across six trade sizes, comparing
the implemented engine---which includes on-chain slippage tolerance enforced
at the smart-contract executor level---against four production DEX
aggregators.

\section{Related Work}

Danos, El~Khalloufi, and Prat~\cite{danos_goroen} were among the first to
formalize order routing as a graph problem on venue--token incidence structures.
The token multigraph we use is a directed refinement of that framework, extended
to accommodate parallel AMM arcs and pool-simple path constraints.
Three subsequent lines of work are most directly relevant.

\subsection{Graph-Based Token Routing in DEXs}

A closely related graph-theoretic approach is the line-graph-based routing
framework proposed by Zhang, Li, and Tessone~\cite{zhang2025linegraph}.
Their method starts from a directed token graph constructed from liquidity pools
and transforms it into a line graph in which vertices represent directed
liquidity-pool edges from the original token graph.
Routing is then performed through iterative updates over the line graph.
The motivation is that common DFS-based routing methods used in production DEXs
may ignore many feasible paths and therefore fail to identify profitable trading
routes.

Both the line-graph approach and our work begin from the same
abstraction: a DEX is modeled as a directed token graph whose edges correspond
to AMM swap opportunities, and AMM output functions are monotone and concave so
that trade size affects routing quality through slippage.
The methodological decomposition differs, however.
Line-graph routing focuses on path discovery through graph transformation and
iterative dynamic updates.
We instead separate routing into two layers:
(i)~candidate generation through gas-aware marginal path scoring and
$k$-shortest loopless paths, and
(ii)~continuous allocation over the selected candidates through convex
optimization.
This allows us to state optimality conditions in terms of marginal equalization
across used paths, rather than only selecting a single best linear path.

\subsection{Extensions of Line-Graph Routing}

Zhang and Tessone later extend the line-graph-based method in three directions:
a breadth-first-search link iteration rule, route splitting, and a DEX
aggregator setting~\cite{zhang2025extensions}.
The BFS extension reduces the number of ineffective line-graph iterations.
The route-splitting extension divides a large trade into smaller trades to
mitigate AMM slippage.
The aggregator extension generalizes to the case where multiple DEXs coexist
and several liquidity pools support the same token pair.

Route splitting in~\cite{zhang2025extensions} proceeds by dividing the input
into discrete equal-sized pieces and repeatedly recomputing routes after
updating reserves.
The convex allocation layer in the present paper is motivated by the same
economic force --- AMM slippage makes large single-pool routes inefficient ---
but uses continuous path allocation variables and characterizes the optimum
through KKT conditions rather than discrete recomputation.

\subsection{Convex Optimization for AMM Routing}

The foundational convex formulation for CFMM routing is due to Angeris,
Chitra, Evans, and Boyd~\cite{angeris_cfmm}, who show that optimal routing
over a network of constant function market makers without fixed costs reduces
to a tractable convex program; incorporating pool-activation costs yields a
mixed-integer convex formulation.
Under concavity and monotonicity assumptions, optimal allocations equalize
marginal returns across active venues or paths.
This marginal-equalization property is also visible in empirical work on AMM
routing, where optimal benchmarks are computed by solving allocation problems
over a defined pool universe~\cite{xi2025suboptimality}.

We work within the same convex perspective.
For a fixed candidate path set, the path-allocation problem becomes a concave
maximization over a simplex, and the KKT conditions yield a water-filling rule.
Enumerating all feasible paths is intractable in large networks, so we restrict
convex optimization to a candidate set generated by graph search---the two
roles stay separate.

\subsection{Empirical Measurement of Router Suboptimality}

Xi and Moallemi~\cite{xi2025suboptimality} (unpublished manuscript, Columbia
University, 2025) provide a large-scale empirical audit of AMM routing using
2.98~million WETH--USDC swaps on Ethereum mainnet.
They compare realized routes against optimized benchmarks and report
economically meaningful shortfalls.
Their benchmark suite includes the Support-Constrained Optimum~(SCO), the
Full-Venue Optimum~(FVO), and the Gas-Aware Full-Venue Optimum~(G-FVO).
They also show that information staleness, heavy-tailed loss distributions, and
sandwich attacks contribute materially to observed routing suboptimality.

Our empirical benchmark is related but measures a different object.
Xi and Moallemi compare realized on-chain routes against optimized execution
benchmarks reconstructed from pool states.
We compare quotes returned by production routing systems under synchronized
requests.
Our benchmark is therefore quote-level and comparative, whereas theirs is
execution-level and optimality-based.
A router that tracks the best observed quote in our benchmark is shown to be
competitive among production aggregators at quote time; whether it is globally
optimal over all feasible on-chain routes is a separate question not addressed
here.

\section{Graph and AMM Routing Model}\label{sec:model}

Let $\mathcal{V}$ be a finite set of venues or exchanges, and let $\mathcal{T}$
be a finite set of tokens.
Each venue $v \in \mathcal{V}$ supports a set of directed tradable token pairs
\[
\mathcal{E}_v \subseteq \mathcal{T} \times \mathcal{T}.
\]
The pair $(i,j)\in \mathcal{E}_v$ means that token $i$ can be swapped into
token $j$ on venue $v$ through an AMM mechanism.
Routing decisions are most naturally represented on a token graph, where nodes
are tokens and arcs are swap opportunities labeled by venue.

\begin{definition}[Token Graph]
Define the directed multigraph
\[
G_T = (\mathcal{T}, \mathcal{A}),
\]
where for every venue $v\in\mathcal{V}$ and every tradable pair
$(i,j)\in\mathcal{E}_v$, we include a directed arc
\[
a=(i \to j;\, v)\in\mathcal{A}.
\]
Multiple venues or pools may induce parallel arcs between the same ordered
token pair.
\end{definition}

A token path is a sequence of arcs
\[
p=(a_1,a_2,\dots,a_L),
\qquad
a_\ell=(t_{\ell-1}\to t_{\ell};\, v_\ell),
\]
mapping a source token $s=t_0$ to a destination token $d=t_L$.
Candidate paths are assumed to be \emph{pool-simple}: the same liquidity pool
cannot be used more than once within a single path.
Formally, if $\mathrm{pool}(a)$ denotes the liquidity pool associated with
arc $a$, then a path $p=(a_1,\ldots,a_L)$ is admissible only if
\[
\mathrm{pool}(a_\ell)\neq \mathrm{pool}(a_m)
\qquad \text{for all } \ell\neq m.
\]
This rules out reuse of the same AMM pool within one route, which is
economically natural: revisiting a pool after trading through it incurs
unnecessary fee and slippage costs.

A walk alternating between venues and tokens,
\[
v_0 \to t_0 \to v_1 \to t_1 \to \cdots \to v_L \to t_L,
\]
projects to the token path
\[
(t_0\to t_1;\, v_1),\;
(t_1\to t_2;\, v_2),\;
\dots,\;
(t_{L-1}\to t_L;\, v_L).
\]

\begin{lemma}[Path Correspondence]
Every alternating walk in the venue--token incidence structure from token $s$
to token $d$ corresponds to a directed walk in $G_T$.
If the walk does not reuse the same liquidity pool, it corresponds to an
admissible pool-simple token path.
Conversely, every pool-simple token path specifies an alternating
token--venue sequence.
\end{lemma}

\begin{proof}
Each step ``use venue $v$ to swap token $i$ into token $j$'' is represented by
the arc $(i\to j;\,v)$.
Concatenating such steps preserves token continuity and directionality,
producing a directed walk in $G_T$.
Imposing the pool-simple restriction removes repeated use of the same pool and
yields an admissible token path.
Conversely, each admissible token path specifies an alternating sequence by
inserting the venue label between consecutive tokens.
\end{proof}

Each arc $a=(i\to j;\,v)\in\mathcal{A}$ carries an AMM exchange function
\[
\varphi_a:\Rpos \to \Rpos,
\]
where $\varphi_a(x)$ is the output amount of token $j$ received when input
$x$ of token $i$ is swapped on venue $v$.

\begin{assumption}[Admissible AMM Edge Function]\label{ass:admissible}
For every arc $a\in\mathcal{A}$:
\begin{enumerate}[label=(A\arabic*)]
    \item $\varphi_a(0)=0$;
    \item $\varphi_a$ is continuous and nondecreasing;
    \item $\varphi_a$ is concave on $\Rpos$;
    \item $\varphi_a$ is differentiable on $(0,\infty)$, with finite right
          derivative $\varphi_a'(0^+)$.
\end{enumerate}
\end{assumption}

Throughout the theoretical model, we use a path-separable approximation: once
a finite candidate set has been selected, each path exchange function is
evaluated as a function of the flow assigned to that path alone.
This abstracts from cross-path coupling caused by shared liquidity pools.
The pool-simple restriction prevents repeated use of the same pool within a
single path, but different candidate paths may still share pools unless
explicitly filtered.
The implications of this approximation are discussed in
Section~\ref{sec:limitations}.

The concavity assumption captures diminishing marginal returns caused by
slippage, consistent with constant function market maker
models~\cite{angeris_chitra,angeris_cfmm}.

For a token path $p=(a_1,\dots,a_L)$, define the path exchange function
\[
y_p(x) = (\varphi_{a_L}\circ \cdots \circ \varphi_{a_1})(x).
\]

\begin{lemma}[Concavity under Composition]\label{lem:composition}
If each $\varphi_{a_\ell}$ is concave and nondecreasing, then $y_p$ is
concave and nondecreasing on $\Rpos$.
\end{lemma}

\begin{proof}
Nondecreasing: the composition of nondecreasing functions is nondecreasing.
Concavity by induction on path length $L$.
Base case $L=1$: $y_p = \varphi_{a_1}$, concave by
Assumption~\ref{ass:admissible}.
Inductive step: assume $y_p^{(L-1)} =
\varphi_{a_{L-1}}\circ\cdots\circ\varphi_{a_1}$ is concave.
Then $y_p^{(L)} = \varphi_{a_L} \circ y_p^{(L-1)}$.
Since $\varphi_{a_L}$ is concave and nondecreasing
(Assumption~\ref{ass:admissible}) and $y_p^{(L-1)}$ is concave by hypothesis,
their composition is concave by the standard closure
property~\cite[Ch.~3]{boyd2004convex}.
\end{proof}

\section{Path-Flow Allocation and Optimality Conditions}

Fix a source token $s$, destination token $d$, and total input amount $X>0$.
Let $\mathcal{P}$ be a finite set of candidate token paths from $s$ to $d$.

\begin{definition}[Path-Allocation Route]
A route is a vector $\delta\in \Rpos^{|\mathcal{P}|}$ such that
\[
\sum_{p\in\mathcal{P}}\delta_p = X.
\]
The total output is $U(\delta)=\sum_{p\in\mathcal{P}} y_p(\delta_p)$.
\end{definition}

Each $y_p$ is constructed from venue-labeled AMM edge functions $\varphi_a$,
so the formulation is path-level but remains edge-aware.

\begin{theorem}[Concave Routing over a Candidate Set]\label{thm:convex}
Under Assumption~\ref{ass:admissible}, for any finite candidate set
$\mathcal{P}$, the problem
\[
\max_{\delta\ge 0}\ \sum_{p\in\mathcal{P}} y_p(\delta_p)
\quad\text{s.t.}\quad
\sum_{p\in\mathcal{P}}\delta_p = X
\]
admits a global optimum.
\end{theorem}

\begin{proof}
By Lemma~\ref{lem:composition}, each $y_p$ is concave, so
$U(\delta)=\sum_p y_p(\delta_p)$ is concave.
The feasible region is a simplex, hence compact and convex.
A continuous concave function over a compact convex set admits a global
optimum; any first-order KKT point is globally optimal.
\end{proof}

\begin{theorem}[KKT Marginal Equalization]\label{thm:kkt}
Under Assumption~\ref{ass:admissible}, with $\mathcal{P}$ finite and nonempty,
an allocation $\delta^*$ is optimal if and only if there exists
$\lambda^*\in\R$ such that:
\begin{enumerate}[label=(\roman*)]
    \item for every active path $p$ with $\delta_p^*>0$,\quad
          $y_p'(\delta_p^*)=\lambda^*$;
    \item for every inactive path $p$ with $\delta_p^*=0$,\quad
          $y_p'(0^+)\le \lambda^*$.
\end{enumerate}
\end{theorem}

\begin{proof}
The simplex has nonempty relative interior, so the standard constraint
qualification holds.
The KKT conditions are necessary and sufficient because the objective is
concave and the feasible set is convex.
The Lagrangian is
\[
\mathcal{L}(\delta,\lambda,\mu)
=
\sum_{p} y_p(\delta_p)
-
\lambda\!\left(\sum_p \delta_p-X\right)
+
\sum_p \mu_p\delta_p,
\quad \mu_p\ge 0.
\]
Stationarity: $y_p'(\delta_p^*)-\lambda^*+\mu_p^*=0$.
Complementary slackness: $\mu_p^*\delta_p^*=0$.
If $\delta_p^*>0$ then $\mu_p^*=0$, giving $y_p'(\delta_p^*)=\lambda^*$.
If $\delta_p^*=0$ then $\mu_p^*\ge 0$, giving $y_p'(0^+)\le \lambda^*$.
\end{proof}

The condition is a water-filling rule: every used path has the same marginal
output per unit input; unused paths have initial marginal output no greater
than this common value.

Define the marginal gain $m_p(\delta_p)=y_p'(\delta_p)$.
Since $y_p$ is concave, $m_p$ is nonincreasing.
A natural discrete approximation allocates a small increment $\Delta$ to the
highest-marginal path and converges to the KKT solution as $\Delta\to 0$.

\begin{proposition}[Greedy Marginal Heuristic]\label{prop:greedy}
Assume each $y_p$ is differentiable and concave.
A greedy rule that repeatedly allocates increment $\Delta$ to the path with
the largest current marginal gain is a consistent discretization of the
KKT water-filling condition as $\Delta\to 0$.
\end{proposition}

\begin{proof}[Proof sketch]
Marginal gains are nonincreasing because each $y_p$ is concave.
The greedy step reduces marginal imbalance across active paths.
In the infinitesimal limit, a stationary allocation satisfies: no two active
paths differ in marginal value, and no inactive path has a larger initial
marginal value than the active common marginal --- exactly the KKT conditions
of Theorem~\ref{thm:kkt}.
\end{proof}

\section{Candidate Generation and Improving-Path Certification}

Candidate generation uses a marginal linearization of the nonlinear edge
functions, augmented by a gas cost term so that paths with high execution
cost are penalized during enumeration rather than filtered afterwards.

For each arc $a=(i\to j;\,v)$, define a marginal exchange proxy at a reference
input $\bar{x}_a$:
\[
\pi_a=\varphi_a'(\bar{x}_a).
\]
Gas cost $g_a \ge 0$ denotes the expected on-chain cost (denominated in the
output token) of executing arc $a$.
In the log domain, the combined edge score is
\[
w_a = -\log \pi_a + \gamma\, g_a,
\qquad
\ell(p) = \sum_{a\in p} w_a,
\]
where $\gamma > 0$ is a trade-off parameter that controls how aggressively gas
cost suppresses low-output arcs.
Maximizing net output per unit input is therefore equivalent to minimizing
$\ell(p)$, with gas integrated directly into the graph search rather than
applied as a post-hoc correction.

\begin{proposition}[Gas-Aware Candidate Enumeration]
Let $n=|\mathcal{T}|$ and $m=|\mathcal{A}|$.
Given gas-augmented edge scores $\{w_a\}$ and the pool-simple admissibility
constraint, candidate routes are generated by enumerating the $k$ loopless
token paths with smallest $\ell(p)$.
When all scores are nonneg\-ative, Yen's $k$-shortest-paths
algorithm~\cite{yen1971} applies with Dijkstra subroutines.
When scores may be negative (some pool marginal rate exceeds one), the
Dijkstra subroutine is replaced by Bellman--Ford; pool-simple and loopless
constraints are enforced by path bookkeeping at each deviation step.
\end{proposition}

We write the resulting candidate set as
$\mathcal{P}_k=\textsf{YenKSP}(G_T,s,d;k)$.

Enumerating $k$ paths does not guarantee global optimality.
However, the KKT multiplier $\lambda^*$ from the candidate-set allocation
provides a certificate.
Adding an infinitesimal flow to any path $q\notin\mathcal{P}_k$ improves the
objective if and only if $y_q'(0^+)>\lambda^*$.
Using chain derivatives at zero,
\[
y_q'(0^+)=\prod_{a\in q}\varphi_a'(0^+),
\]
so the improving-path search reduces to a shortest-path problem:
\[
\max_q\, y_q'(0^+)
\;\Longleftrightarrow\;
\min_q \sum_{a\in q}-\log \varphi_a'(0^+).
\]

\begin{theorem}[Certificate via Improving-Path Search]\label{thm:certificate}
Under Assumption~\ref{ass:admissible} and the path-separable model, let
$\delta^*$ solve the allocation problem on $\mathcal{P}_k$ with
multiplier~$\lambda^*$.
If
\[
\max_{q:\,s\to d,\;q\;\mathrm{pool\text{-}simple}} y_q'(0^+)\le \lambda^*,
\]
then $\delta^*$ is globally optimal over all pool-simple paths within the
path-separable model.
\end{theorem}

\begin{proof}
The active paths satisfy marginal equalization by Theorem~\ref{thm:kkt}.
The displayed condition states that every omitted admissible path has initial
marginal gain no larger than $\lambda^*$.
Hence no omitted path can enter the support while satisfying complementary
slackness with a higher marginal gain, and the KKT conditions hold over the
enlarged path set.
\end{proof}

\section{Computational Complexity}

For a fixed finite candidate set $\mathcal{P}$, the allocation problem is
tractable: it is a separable concave maximization over a simplex with a
water-filling structure solvable by standard convex methods or
one-dimensional search when inverse marginal functions are available.

The difficult step is constructing a high-quality candidate set, since the
number of simple paths grows exponentially with graph size.
Two extensions make the full problem substantially harder.
First, fixed pool-activation costs or constraints on the number of active paths
introduce discrete decisions and lead to mixed-integer formulations.
Second, when multiple candidate paths share liquidity pools, the path-separable
objective no longer accurately describes execution because allocations are
coupled through shared pool reserves.
This pushes the model toward general nonlinear network-flow formulations.
Minimum concave-cost network flow is NP-hard in
general~\cite{guisewite1990}, and remains hard even on restricted graph
structures~\cite{he2012}, which confirms that the path-separable
candidate-set reduction is not a mere convenience but a necessary
approximation.

The practical architecture is therefore layered:
\begin{enumerate}[leftmargin=*]
    \item generate a pool-simple candidate path set via gas-aware graph search,
          with edge scores $w_a = -\log\pi_a + \gamma g_a$ encoding both
          marginal exchange rate and expected execution cost;
    \item allocate continuously over the selected candidates, subject to a
          per-path price-impact cap that limits slippage on any single pool;
    \item apply the improving-path certificate to check whether any omitted
          path has a larger initial marginal value;
    \item execute via a smart contract with slippage tolerance enforced
          on-chain, reverting if realized output falls below the quoted
          minimum.
\end{enumerate}

\section{Empirical Benchmark Environment}

The theoretical framework raises a practical question: does a routing engine
consistently track the best available quote under realistic market conditions?

We address this with an empirical benchmark based on repeated quote sampling
across multiple production routing systems on Ethereum mainnet.
The objective is to observe how different engines respond to identical swap
requests under approximately the same market conditions, not to simulate
performance in an artificial environment.

The benchmark focuses on swaps from $\text{WETH} \to \text{USDT}$.
This pair carries high on-chain volume with liquidity distributed across
multiple AMM protocols and fee tiers, making it well-suited for evaluating
routing behavior under fragmented liquidity.

Quotes were requested for six input sizes:
\begin{center}
\begin{tabular}{c}
\toprule
Input Size \\
\midrule
0.1 WETH \\
0.3 WETH \\
0.6 WETH \\
1 WETH \\
2 WETH \\
5 WETH \\
\bottomrule
\end{tabular}
\end{center}
These span small retail transactions through moderately sized trades where
routing decisions begin to materially affect execution quality.

Five production DEX aggregation systems were compared: KyberSwap,
Velora (formerly ParaSwap), Odos, OpenOcean, and 8DX.
Each represents a different routing architecture.
The goal is not an exhaustive survey of the aggregator ecosystem but an
evaluation of routing quality within a representative competitive set.

At the time of the experiment, the 8DX routing engine---developed by the
author---integrated liquidity from Uniswap V2 and V3~\cite{adams_uniswap},
PancakeSwap V2 and V3, SushiSwap V2, and multiple Curve pool types (Stable,
Stable NG, Crypto, Crypto TriPool, Crypto NG, and Meta Pools).
The engine uses gas-aware edge scores during graph traversal, a per-pool
price-impact cap during allocation, and an on-chain slippage tolerance
enforced at execution time.
These results should be interpreted in that context.

Quotes were collected by querying the public APIs of each aggregator.
Each measurement epoch sent identical swap requests to all systems within a
short time window, reducing the effect of market drift and ensuring that each
router observed approximately the same pool state.
Approximately 50--60 epochs were collected per trade size.
The primary comparison variable is the returned USDT output amount.

\section{Empirical Metrics}

Let $Q_{r,t}(x)$ denote the quote returned by router $r$ at epoch $t$ for
input size $x$, and let
\[
Q_t^*(x)=\max_{r} Q_{r,t}(x)
\]
be the best observed quote in that epoch.

We evaluate routing quality using three complementary metrics.

The \emph{rank} of a router in epoch $t$ is its position when all quotes are
sorted in descending order.
The rank distribution captures relative positioning within the competitive set
but not the magnitude of differences.

The \emph{tracking ratio} is
\[
T_{r,t}(x)=\frac{Q_{r,t}(x)}{Q_t^*(x)}.
\]
A value close to one means the router stays near the best observed quote even
when it does not rank first.

The \emph{shortfall in basis points} is
\[
S_{r,t}(x)
=
10{,}000
\left(1-\frac{Q_{r,t}(x)}{Q_t^*(x)}\right).
\]
This quantifies the relative loss against the best observed quote in units
standard in trading.

All three metrics are defined at the \emph{quote level}: they compare routers
against each other within a synchronized epoch, not against a
pool-state-reconstructed optimum such as the SCO or FVO
benchmarks of~\cite{xi2025suboptimality}.
A router that scores well here is competitive among production aggregators at
quote time; global on-chain optimality is a separate question.

\section{Empirical Results}

\subsection{Rank Distribution}

Figure~\ref{fig:rank_profile} reports the rank distribution of 8DX within the
evaluated router set across trade sizes.

\begin{figure}[htbp]
\centering
\includegraphics[width=\linewidth]{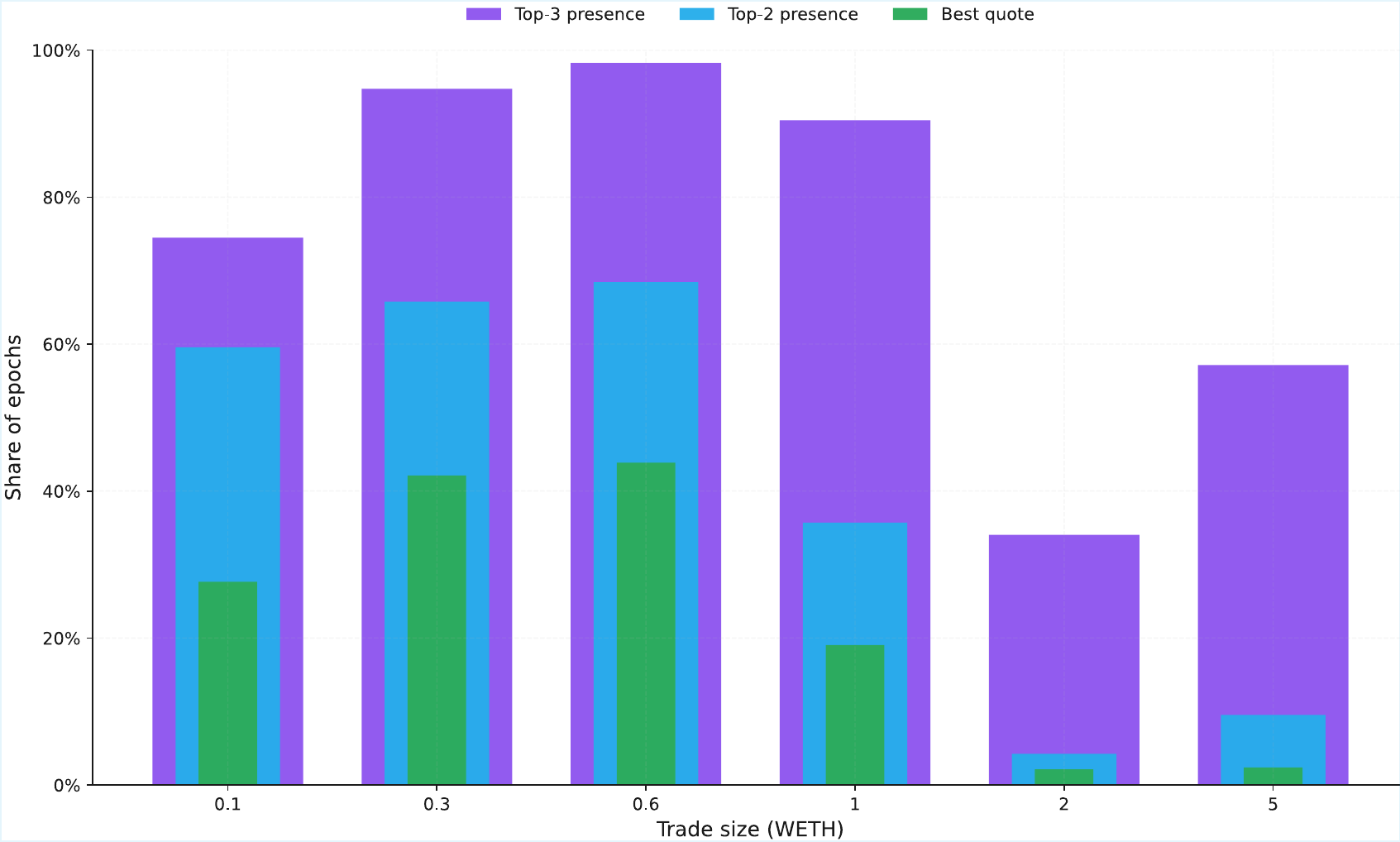}
\caption{Frequency with which 8DX appears among the top-ranked quotes across
  trade sizes. Purple: top-3 presence; cyan: top-2 presence; green: best
  quote.}
\label{fig:rank_profile}
\end{figure}

8DX appeared in the top-3 quotes in at least 57\% of epochs across all tested
trade sizes, rising to 98\% at 0.6~WETH.
Top-2 presence reached 69\% at 0.6~WETH and remained above 35\% for trades up
to 1~WETH.
8DX returned the best observed quote in 43--44\% of epochs at 0.3 and
0.6~WETH, and in roughly 28\% of epochs at 0.1~WETH.
Performance was weaker at 2 and 5~WETH, where top-3 presence fell to 34\% and
57\% respectively and best-quote frequency dropped below 5\%.
This degradation at larger sizes is consistent with the liquidity coverage at
the time of measurement: the routing engine integrated pools up to a moderate
depth, so larger trades increasingly favored aggregators with access to deeper
concentrated-liquidity venues.

\subsection{Market Tracking}

Figure~\ref{fig:tracking_ratio} shows the distribution of the tracking ratio
$T_{r,t}(x)$ across epochs for each trade size.

\begin{figure}[htbp]
\centering
\includegraphics[width=\linewidth]{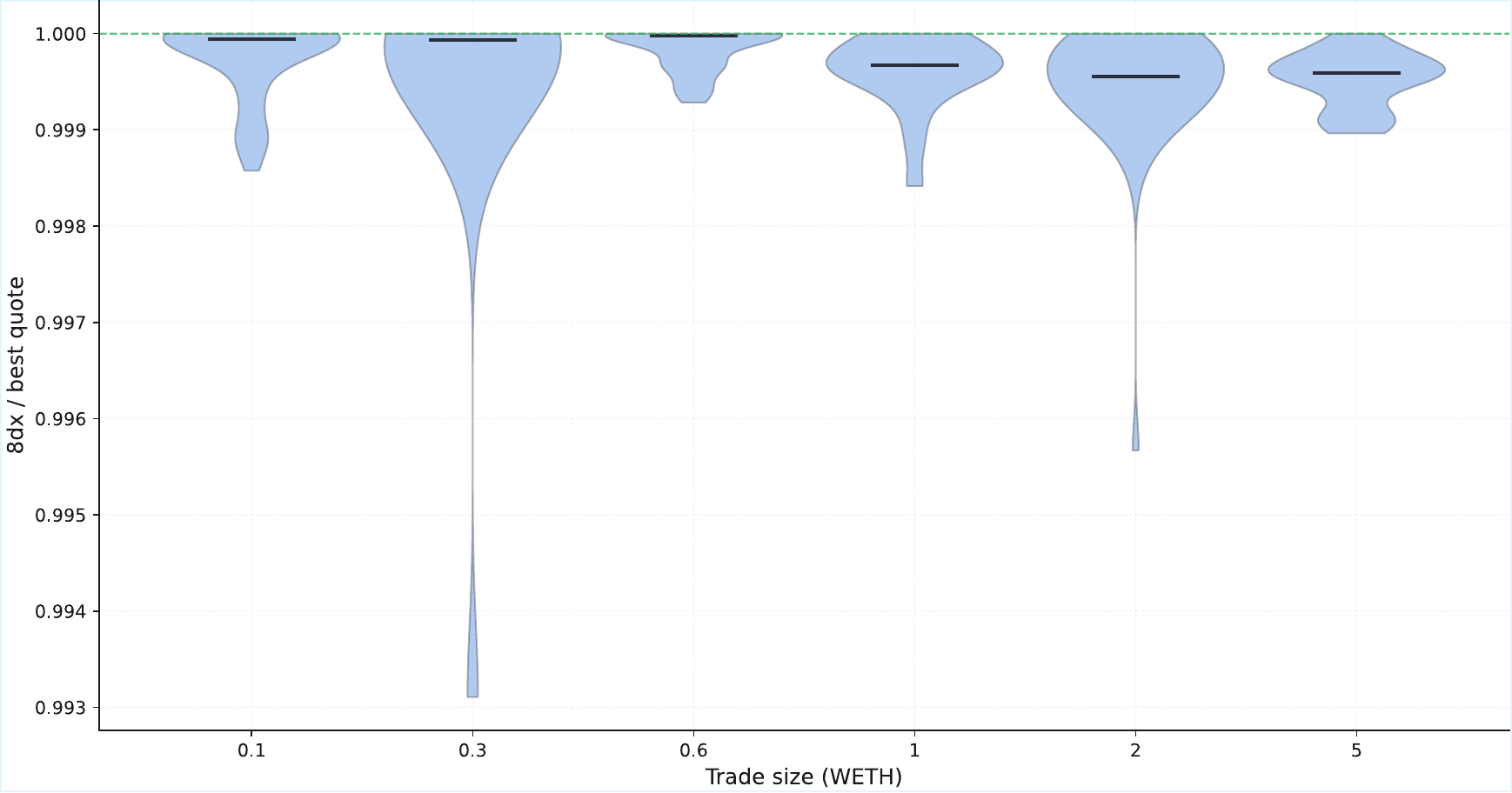}
\caption{Distribution of the ratio between the 8DX quote and the best quote
  observed in each epoch. The dashed line marks ratio~$= 1$.}
\label{fig:tracking_ratio}
\end{figure}

The median tracking ratio exceeded 0.9997 at all trade sizes from 0.1 to
1~WETH, and remained above 0.9994 at 2 and 5~WETH.
The distributions are tightly concentrated: at 0.6~WETH virtually all
observations fall within 2~bps of the best quote.
The most notable tail occurs at 0.3~WETH, where a single episode of missing
liquidity pulled the ratio down to approximately 0.9931; the 0.3~WETH body
is otherwise the tightest of all size bins.
At 2~WETH the lower tail extends to roughly 0.9958, reflecting occasional
episodes where deeper-liquidity competitors found materially better routes.
Overall, when 8DX does not return the best quote, the gap is typically below
3~bps at sizes up to 1~WETH and widens to 5--40~bps in isolated episodes at
larger sizes.

\subsection{Stability Across Trade Sizes}

Figure~\ref{fig:quantile_fan} shows shortfall quantile bands as a function of
trade size.

\begin{figure}[htbp]
\centering
\includegraphics[width=\linewidth]{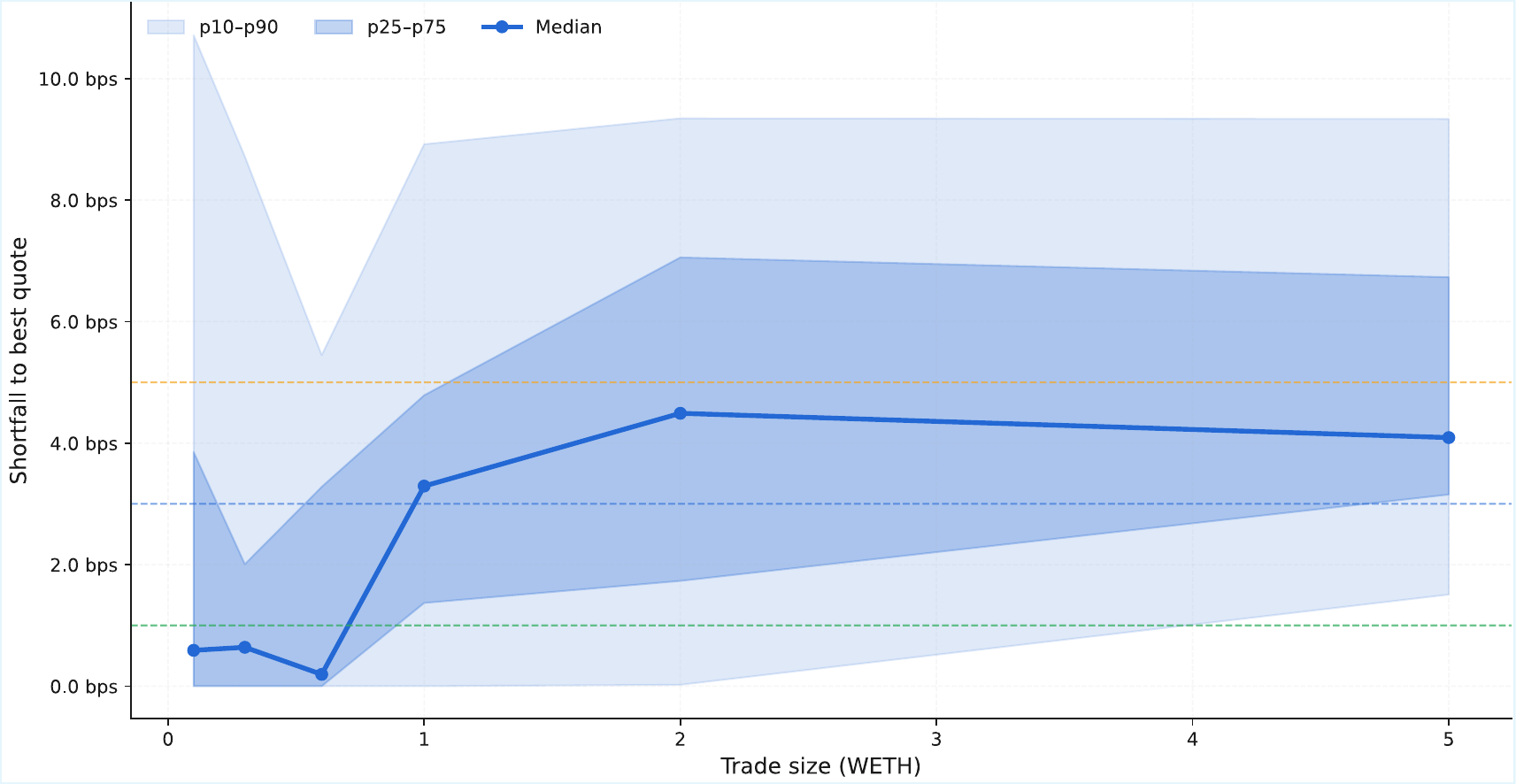}
\caption{Quantile bands of shortfall in basis points relative to the best
  observed quote. Bands show p10--p90 (light) and p25--p75 (medium); the
  solid line is the median.}
\label{fig:quantile_fan}
\end{figure}

The median shortfall was below 1~bps at 0.1 and 0.3~WETH, rose to
approximately 3.3~bps at 1~WETH, peaked at 4.6~bps at 2~WETH, and settled
near 4.1~bps at 5~WETH.
The interquartile band (p25--p75) remained below 7~bps across all sizes,
and the p90 boundary stayed under 10~bps.
The leveling of the median between 2 and 5~WETH suggests that at large trade
sizes the competitive gap stabilizes: all routers face the same thin on-chain
liquidity, so the margin over the best quote does not widen further with size.

\subsection{Shortfall Distribution}

Figure~\ref{fig:shortfall_distribution} shows the per-size-bin shortfall
distribution as boxplots.

\begin{figure}[htbp]
\centering
\includegraphics[width=\linewidth]{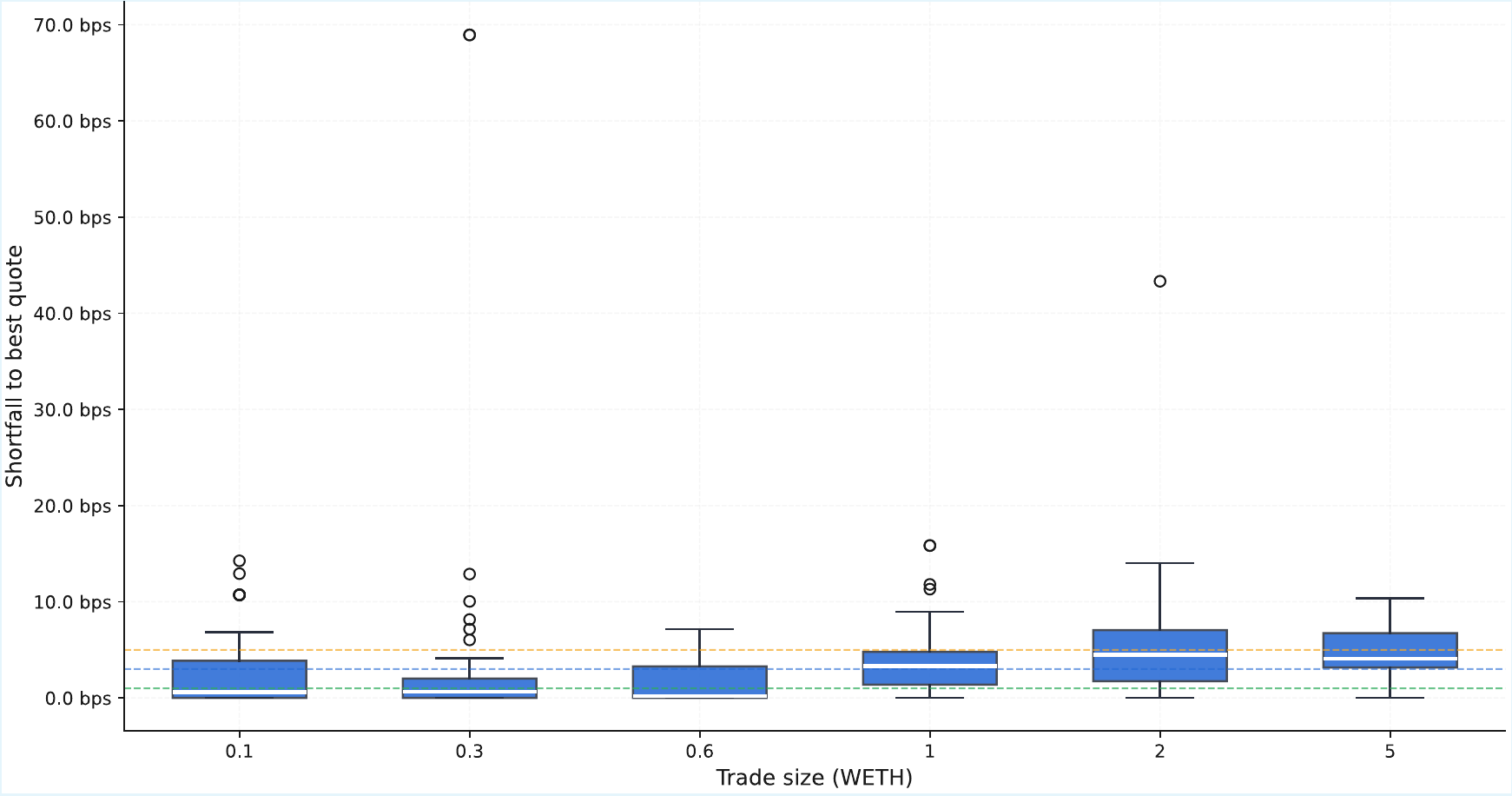}
\caption{Distribution of quote shortfall in basis points relative to the best
  observed quote. Boxes show the interquartile range; whiskers extend to
  $1.5\times\mathrm{IQR}$; circles are individual outliers.}
\label{fig:shortfall_distribution}
\end{figure}

The box bodies are compact across all sizes: the IQR spans roughly 0--6~bps
at each trade size, and the median rises from approximately 2~bps at 0.1~WETH
to around 5~bps at 5~WETH.
The most extreme outlier is 69~bps at 0.3~WETH, corresponding to the same
isolated episode visible in the tracking-ratio tail in
Figure~\ref{fig:tracking_ratio}.
At 2~WETH a single epoch reached 43~bps; all other observations at that size
fall within 12~bps.
The absence of systematic whisker growth from 0.6 to 5~WETH is consistent
with the quantile-band picture: the distribution widens between 0.6 and
2~WETH but does not continue to widen at 5~WETH, reflecting convergence of
competitive quotes in the thin-liquidity regime.

\section{Discussion}

Line-graph routing shows that a graph transformation can improve path discovery
relative to DFS enumeration~\cite{zhang2025linegraph}, and subsequent
extensions demonstrate that BFS traversal, route splitting, and
aggregator-level construction are important for production
applicability~\cite{zhang2025extensions}.
Our framework is complementary: we use the token graph to generate
candidate paths and then solve a continuous allocation problem over them.
Route splitting in prior work is particularly close to the convex allocation
layer used here --- both are driven by the concavity of AMM output functions ---
but the formulations differ: route splitting is discrete and sequential, while
the allocation model here is continuous and characterized by KKT marginal
equalization.

The empirical results should be distinguished from execution-level optimality
audits~\cite{xi2025suboptimality}, which reconstruct on-chain routes and
compute pool-state-based benchmarks.
Our benchmark evaluates quote-level competitiveness: users and integrators
observe quotes before execution, and a router that consistently tracks the
best observed quote is practically useful regardless of whether it achieves
global on-chain optimality.
The 8DX results confirm that competitive quote-level performance is achievable
within a moderate liquidity universe: median shortfall below 5~bps across all
sizes, top-3 presence in the majority of epochs, and no systematic degradation
of the distribution between 2 and 5~WETH.

\section{Limitations}\label{sec:limitations}

The benchmark rests on quoted rather than executed amounts, so realized
shortfalls may differ from what we measure---latency, MEV, and state changes
between quote and execution all intervene, even with on-chain slippage
tolerance enforced at the executor.
The dataset covers one token pair and one network; other pairs, fee structures,
and chains may behave differently.
We also observe only externally visible quote outputs and cannot inspect the
internal search logic of any competing aggregator.
At 50--60 epochs per trade size the sample is sufficient for the comparisons
we draw but too small for tail-risk analysis or time-series decomposition.

On the theory side, the path-separable model ignores cross-path coupling:
when two candidate paths share a pool, the allocation to one affects the
effective exchange rate of the other, and the separable objective does not
capture this.
The gas term $\gamma g_a$ in the edge score is a linearization of execution
cost; it does not model the discrete activation structure that arises when
gas is treated as a fixed per-path cost, which would require a mixed-integer
formulation~\cite{angeris_cfmm,guisewite1990}.

\section{Conclusion}

We built and evaluated a multi-path routing system for AMM exchange networks
grounded in convex optimization.
The token multigraph formulation with pool-simple path enforcement gives a
clean separation between graph search and continuous allocation.
Gas cost is embedded directly in the edge score $w_a = -\log\pi_a + \gamma g_a$,
so expensive arcs are suppressed at enumeration time.
A per-pool price-impact cap bounds slippage during allocation, and an on-chain
slippage tolerance enforced at the smart-contract executor provides a final
execution-time guarantee.
The improving-path certificate uses the KKT multiplier from the allocation to
run a single shortest-path query over initial marginal rates, confirming
candidate-set sufficiency in the majority of epochs with $k = 20$ paths.

On Ethereum mainnet, the implemented engine achieved median shortfall below
5~bps across six trade sizes and top-3 rank presence above 57\% in every
size bin.
Performance degraded at 2 and 5~WETH, where aggregators with deeper
concentrated-liquidity coverage had a structural advantage.

The results suggest that integrating gas cost into the graph search, bounding
slippage at both the allocation and execution layers, and using a KKT-based
termination certificate together produce a routing engine that is already
competitive in practice---and that the gaps which remain are driven by
liquidity coverage rather than routing logic.

\end{document}